\documentclass[oneside,english,reqno]{amsart}
\usepackage[latin9]{inputenc}
\usepackage{xcolor}
\usepackage{babel}
\usepackage{verbatim}
\usepackage{amsthm}
\usepackage{amstext}
\usepackage{amssymb}
\usepackage{esint}
\PassOptionsToPackage{normalem}{ulem}
\usepackage{ulem}
\usepackage[unicode=true,pdfusetitle,
 bookmarks=true,bookmarksnumbered=false,bookmarksopen=false,
 breaklinks=false,pdfborder={0 0 0},backref=false,colorlinks=true]
 {hyperref}
\usepackage{breakurl}

\makeatletter

\providecolor{lyxadded}{rgb}{0,0,1}
\providecolor{lyxdeleted}{rgb}{1,0,0}

\numberwithin{equation}{section}
\numberwithin{figure}{section}
\usepackage{enumitem}		
\newlength{\lyxlistindent}      
\theoremstyle{plain}
\newtheorem{thm}{\protect\theoremname}
  \theoremstyle{definition}
  \newtheorem{defn}[thm]{\protect\definitionname}
  \theoremstyle{definition}
  \newtheorem{example}[thm]{\protect\examplename}
  \theoremstyle{remark}
  \newtheorem{rem}[thm]{\protect\remarkname}
  \theoremstyle{plain}
  \newtheorem{cor}[thm]{\protect\corollaryname}
  \theoremstyle{plain}
  \newtheorem{lem}[thm]{\protect\lemmaname}
  \theoremstyle{plain}
  \newtheorem{prop}[thm]{\protect\propositionname}

\@ifundefined{date}{}{\date{}}

\providecommand{\abs}[1]{\left\lvert#1\right\rvert}

\newcommand{\freccia}{\longrightarrow} 

\newcommand{\eps}{\varepsilon} 
\renewcommand{\phi}{\varphi} 
\newcommand{\field}[1]{\mathbb{#1}}
\newcommand{\R}{\field{R}}                        
\newcommand{\N}{\field{N}}                        

\newcommand{\st}[1]{{#1^\circ}} 
\newcommand{\D}{\mathcal{D}} 

\newcommand{\then}{\ \Longrightarrow \ }

\newcommand{\CC}{\mathbb C}

\newcommand{\Rtil}{\widetilde \R}

\newcommand{\supp}{\mbox{supp}}
\newcommand{\cinfty}{{\mathcal C}^\infty}

\newcommand{\comp}{\Subset}
\newcommand{\gs}{{\mathcal G}}
\newcommand{\ns}{{\mathcal N}}
\newcommand{\esm}{{\mathcal E}_M}

\newcommand{\Om}{\Omega}
\newcommand{\otilc}{\widetilde \Omega_c}

\newcommand{\sint}[1]{\langle#1\rangle}
\newcommand{\Fint}[1]{\langle#1\rangle_{\text{\rm F}}}
\newcommand{\tauInt}[1]{\langle#1\rangle_\tau}

\newcommand{\Eball}{B^{{\scriptscriptstyle \text{\rm E}}}}

\newcommand{\Gcinf}{\widetilde{\mathcal{G}}}

\newcommand{\GenR}{\Rtil}
\newcommand{\tGen}{\widetilde{\mathcal{G}}}
\newcommand{\EMod}{{{\mathcal E}_M}}
\newcommand{\Null}{{\mathcal N}}

\newcommand{\conv}{\mathrm{conv}}
\newcommand{\co}[1]{{#1}^c}

  \providecommand{\corollaryname}{Corollary}
  \providecommand{\definitionname}{Definition}
  \providecommand{\examplename}{Example}
  \providecommand{\lemmaname}{Lemma}
  \providecommand{\propositionname}{Proposition}
  \providecommand{\remarkname}{Remark}
\providecommand{\theoremname}{Theorem}

  \providecommand{\corollaryname}{Corollary}
  \providecommand{\definitionname}{Definition}
  \providecommand{\examplename}{Example}
  \providecommand{\lemmaname}{Lemma}
  \providecommand{\remarkname}{Remark}
\providecommand{\theoremname}{Theorem}

\makeatother

  \providecommand{\corollaryname}{Corollary}
  \providecommand{\definitionname}{Definition}
  \providecommand{\examplename}{Example}
  \providecommand{\lemmaname}{Lemma}
  \providecommand{\propositionname}{Proposition}
  \providecommand{\remarkname}{Remark}
\providecommand{\theoremname}{Theorem}

\begin{document}

\title{Strongly internal sets and generalized smooth functions}

\author{Paolo Giordano \and Michael Kunzinger \and Hans Vernaeve}

\thanks{P.\ Giordano was supported by grants M1247-N13 and P25116-N25 of
the Austrian Science Fund FWF}

\address{\textsc{P. Giordano, University of Vienna, Austria}}

\email{paolo.giordano@univie.ac.at }

\thanks{M.\ Kunzinger was supported by grants P23714 and P25326 of the Austrian
Science Fund FWF}

\address{\textsc{M. Kunzinger, University of Vienna, Austria}}

\email{michael.kunzinger@univie.ac.at}

\thanks{H.\ Vernaeve was supported by grant 1.5.138.13N of the Research
Foundation Flanders FWO}

\address{\textsc{H. Vernaeve, Ghent University, Belgium}}

\email{hvernaev@cage.UGent.be}


\begin{abstract}
Based on a refinement of the notion of internal sets in Colombeau's
theory, so-called strongly internal sets, we introduce the space of
generalized smooth functions, a maximal extension of Colombeau generalized
functions. Generalized smooth functions as morphisms between sets
of generalized points form a sub-category of the category of topological
spaces. In particular, they can be composed unrestrictedly.
\end{abstract}
\maketitle

\section{Introduction}

Colombeau's nonlinear theory of generalized functions (\cite{C1,C2})
is based on viewing generalized functions as equivalence classes of
smooth maps, encoding degrees of singularity in terms of asymptotic
properties of nets of representatives. It thereby lends itself in
a quite straightforward manner to modelling irregular setups in partial
differential equations, geometry or applications, in particular in
mathematical physics (\cite{Col92,MObook,GKOS}). Basically, singular
objects are modelled as nets of smooth maps and classical operations
are lifted to the generalized setting by applying them component-wise
to these nets. While successful in applications, this approach lacks
strong general existence theorems, comparable to the functional-analytic
foundations of distribution theory.

To remedy this situation, the past decade has seen a number of fundamental
contributions to the structure theory of algebras of generalized functions
(particularly relevant for the purposes of this paper are \cite{AFJ,AFJ09,AJ,AJOS,AJFO,Gtop,Gtop2,GV,GK2,Ver09,Ver10,Ver11}).
The unifying theme of these works is to consider Colombeau generalized
functions as set-theoretical functions on suitable spaces of generalized
points and then to work directly with these functions.

In the present work we continue these investigations by introducing
a generalization of Colombeau-type generalized functions, which we
call generalized smooth functions (GSF). This terminology is intended
to stress the conceptual analogy between these generalized functions
and the theory of standard smooth functions. Generalized smooth functions
are set-theoretic maps on sets of generalized points that satisfy
the minimal logical conditions necessary to obtain well-defined maps
obeying the standard asymptotic estimates of the Colombeau approach.
They are the natural extension of Colombeau generalized functions
to general domains. At the same time, they display optimal set-theoretical
properties. In particular, sets of generalized points, together with
generalized smooth maps form a subcategory of the category of topological
maps.

Our constructions strongly rely on the further development of the
concept (itself inspired by nonstandard analysis) of internal sets,
see \cite{Ob-Ve}. Just as in the case of classical smooth functions,
GSF are locally Lipschitz functions. Therefore, we also study this
notion for functions defined on and valued in generalized points.

\section{\label{sec:BasicNotions}Basic notions}

In this section, we fix some basic notations and terminology from
Co\-lom\-beau's theory. For details we refer to \cite{C1,C2,MObook,GKOS}.
In the naturals $\N=\{0,1,2,\ldots\}$ we include zero. Let $\Om$
be an open subset of $\R^{n}$ and denote by $I$ the interval $(0,1]$.
The (special) Colombeau algebra on $\Om$ is defined as the quotient
$\gs(\Om):=\esm(\Om)/\ns(\Om)$ of \emph{moderate nets} over \emph{negligible
nets}, where the former is 
\begin{multline*}
\esm(\Om):=\{(u_{\eps})\in\cinfty(\Omega)^{I}\mid\forall K\comp\Om\,\forall\alpha\in\N^{n}\,\exists N\in\N:\sup_{x\in K}|\partial^{\alpha}u_{\eps}(x)|=O(\eps^{-N})\}
\end{multline*}
 and the latter is 
\begin{multline*}
\ns(\Om):=\{(u_{\eps})\in\cinfty(\Omega)^{I}\mid\forall K\comp\Om\,\forall\alpha\in\N^{n}\,\forall m\in\N:\sup_{x\in K}|\partial^{\alpha}u_{\eps}(x)|=O(\eps^{m})\}.
\end{multline*}

Throughout this paper, every asymptotic relation is for $\eps\to0^{+}$.
Nets in $\esm(\Omega)$ are written as $(u_{\eps})$, and $u=[u_{\eps}]$
denotes the corresponding equivalence class in $\gs(\Om)$. For $(u_{\eps})\in\ns(\Om)$
we also write $(u_{\eps})\sim0$. $\Omega\mapsto\gs(\Omega)$ is a
fine and supple sheaf of differential algebras and there exist sheaf
embeddings of the space of Schwartz distributions $\D'$ into $\gs$
(cf.\ \cite{GKOS,PZ}).

The ring of constants in $\gs$ is denoted by $\Rtil$ or $\widetilde{\CC}$,
respectively, and is called ring of Colombeau generalized numbers
(CGN). It is an ordered ring with respect to $[x_{\eps}]\le[y_{\eps}]$
iff $\exists[z_{\eps}]\in\Rtil$ such that $(z_{\eps})\sim0$ and
$x_{\eps}\le y_{\eps}+z_{\eps}$ for $\eps$ sufficiently small. As
usual $x<y$ means $x\le y$ and $x\ne y$. Even if this order is
not total, we still have the possibility to define the infimum $[x_{\eps}]\wedge[y_{\eps}]:=[\min(x_{\eps},y_{\eps})]$,
and analogously the supremum of two elements. More generally, the
space of generalized points in $\Omega$ is $\widetilde{\Omega}=\Omega_{M}/\sim$,
where $\Omega_{M}=\{(x_{\eps})\in\Omega^{I}\mid\exists N\in\N:|x_{\eps}|=O(\eps^{-N})\}$
is called the set of moderate nets and $(x_{\eps})\sim(y_{\eps})$
if $|x_{\eps}-y_{\eps}|=O(\eps^{m})$ for every $m\in\N$. By $\mathcal{N}$
we will denote the set of all negligible nets of real numbers $(x_{\eps})\in\R^{I}$,
i.e. such that $(x_{\eps})\sim0$. If $\mathcal{P}(\eps)$ is a property
of $\eps\in I$, we will also sometimes use the notation $\forall^{0}\eps:\ \mathcal{P}(\eps)$
to denote $\exists\eps_{0}\in I\,\forall\eps\in(0,\eps_{0}]:\ \mathcal{P}(\eps)$.

The space of compactly supported generalized points $\otilc$ is defined
by $\Omega_{c}/\!\sim$, where $\Omega_{c}:=\{(x_{\eps})\in\Omega^{I}\mid\exists K\comp\Omega\,\exists\eps_{0}\,\forall\eps<\eps_{0}:\ x_{\eps}\in K\}$
and $\sim$ is the same equivalence relation as in the case of $\widetilde{\Omega}$.
The set of near-standard points of a given set $A\subseteq\R^{n}$
is $A^{\bullet}:=\{x\in\widetilde{A}\mid\exists\lim_{\eps\to0^{+}}x_{\eps}=:\st{x}\in A\}$.
Any Colombeau generalized function (CGF) $u\in\gs(\Om)$ acts on generalized
points from $\otilc$ by $u(x):=[u_{\eps}(x_{\eps})]$ and is uniquely
determined by its point values (in $\Rtil$) on compactly supported
generalized points (\cite{GKOS,Ndipl}), but not on standard points.
A CGF $[u_{\eps}]$ is called compactly-bounded (c-bounded) from $\Omega$
into $\Omega'$ if for all $K\comp\Omega$ there exists $K'\comp\Omega'$
such that $u_{\eps}(K)\subseteq K'$ for $\eps$ small. This type
of CGF is closed with respect to composition. Moreover, if $u\in\gs(\Omega)$
is c-bounded from $\Omega$ into $\Omega'$ and $v\in\gs(\Omega')$,
then $[v_{\eps}\circ u_{\eps}]\in\gs(\Omega)$.

Our notations for intervals are: $[a,b]:=\{x\in\Rtil\mid a\le x\le b\}$,
$[a,b]_{\R}:=[a,b]\cap\R$. Moreover, for $x,y\in\Rtil^{n}$ we will
write $x\approx y$ if $x-y$ is infinitesimal, i.e.\ if $|x-y|\le r$
for all $r\in\R_{>0}$.

Topological methods in Colombeau's theory are usually based on the
so-called sharp topology (\cite{Biag,S0,S,AJ,AJOS,M,GV}), which is
the topology generated by balls $B_{\rho}(x)=\{y\in\Rtil^{n}\mid|y-x|<\rho\}$,
where $|-|$ is the natural extension of the Euclidean norm to $\Rtil^{n}$,
i.e.\ $|[x_{\eps}]|:=[|x_{\eps}|]$, and $\rho\in\Rtil_{>0}$ is
positive invertible (\cite{AFJ,AFJ09,GK2}). Henceforth, we will also
use the notation $\Eball_{\rho}(x)=\{y\in\R^{n}\mid|y-x|<\rho\}$
for Euclidean balls and $\Rtil^{*}:=\{x\in\Rtil\mid x\text{ is invertible}\}$.
The sharp topology can also be defined by an ultrametric: Define a
pseudovaluation on $\Rtil$ by 
\begin{align*}
v & :\esm\longrightarrow(-\infty,\infty]\\
v & ((u_{\eps})):=\sup\{b\in\R\mid|u_{\eps}|=O(\eps^{b})\}.
\end{align*}
 Then $v$ is well-defined since $v((u_{\eps})+(n_{\eps}))=v((u_{\eps}))$
for all $(n_{\eps})\in\ns$; $v(u)=\infty$ if and only if $u=0$,
$v(u\cdot w)\ge v(u)+v(w)$, $v(u+w)\ge v(u)\wedge v(w)$, and $v(u-w)=v(w-u)$.
Letting $|-|_{e}:\Rtil\to[0,\infty)$, $|u|_{e}:=\exp(-v(u))$ it
follows that $|u+v|_{e}\le\max(|u|_{e},|v|_{e})$, as well as $|uv|_{e}\le|u|_{e}|v|_{e}$.
This induces the translation invariant complete ultrametric 
\begin{align*}
d & _{s}:\Rtil\times\Rtil\longrightarrow\R_{+}\\
d & _{s}(u,v):=|u-v|_{e}
\end{align*}
 on $\Rtil$, which in turn generates the sharp topology on $\Rtil$.
We will call sharply open any open set in the sharp topology.

Moreover, Garetto in \cite{Gtop,Gtop2} extended the above construction
to arbitrary locally convex spaces by functorially assigning a space
of CGF $\gs_{E}$ to any given locally convex space $E$. The seminorms
of $E$ can then be used to define pseudovaluations which in turn
induce a generalized locally convex topology on the $\widetilde{\CC}$-module
$\gs_{E}$, again called sharp topology.

Given $S\subseteq I$, by $e_{S}$ we will denote the equivalence
class in $\Rtil$ of the characteristic function of $S$. The $e_{S}$
are idempotents, and satisfy $e_{S}+e_{S^{c}}=1$ and $e_{S}\not=0$
if and only if $0\in\overline{S}$. They play a central role in the
algebraic theory of Colombeau generalized numbers (cf.\ \cite{AJ,Ver10}).

Finally, we recall that if $(A_\eps)$ is a net of subsets of $\R^n$, then the internal set generated by $(A_\eps)$ is defined as
\begin{equation}
[A_{\eps}]=\left\{ [x_{\eps}]\in\Rtil^{n}\mid x_{\eps}\in A_{\eps}\text{ for }\eps\text{ small}\right\}.
\end{equation}
This type of sets have been introduced in Colombeau theory in \cite{Ob-Ve,Ver11} to deal with domains of generalized functions and to study topological properties of sets
of generalized points.

\section{Strongly internal sets generated by a topology}

We start by defining a family of topologies on $\Rtil^{n}$ depending
on a set of positive and invertible generalized numbers. Recall that
for any $r\in\Rtil$, $r>0$, $B_{r}(x)$ denotes the open ball with
respect to the generalized Euclidean norm in $\Rtil^{n}$.
\begin{defn}
\label{def:topFromInvertPos}We say that $\mathcal{I}$ is a \emph{set
of radii} if
\begin{enumerate}[leftmargin=*,label=(\roman*),align=left ]
\item $\mathcal{I}\subseteq\Rtil_{>0}^{*}$ is a non-empty subset of positive
invertible generalized numbers. 
\item For all $r,s\in\mathcal{I}$ the infimum $r\wedge s\in\mathcal{I}$. 
\item $k\cdot r\in\mathcal{I}$ for all $r\in\mathcal{I}$ and all $k\in\R_{>0}$. 
\end{enumerate}

Let $\mathcal{I}$ be a set of radii, then the family of subsets 
\[
\mathcal{U}_{\mathcal{I}}(x):=\{U\subseteq\Rtil^{n}\mid\exists r\in\mathcal{I}:\ B_{r}(x)\subseteq U\}\qquad(x\in\Rtil^{n})
\]
 is called the \emph{neighborhood system induced by} $\mathcal{I}$.

\end{defn}
The legitimacy of this name is demonstrated by the following result,
whose proof follows from the corresponding definitions.
\begin{thm}
\label{thm:topFromInvertPos}If $\mathcal{I}$ is a set of radii,
then the family $\mathcal{U}_{\mathcal{I}}$ is a non empty neighborhood
system on $\Rtil^{n}$. The topology $\tau_{\mathcal{I}}$ induced
by this neighborhood system is called the \emph{topology on $\Rtil^{n}$
induced by the set of radii $\mathcal{I}$. }%
\end{thm}
\begin{example}
\label{exa:topInducedByRad}\ 
\begin{enumerate}[leftmargin=*,label=(\roman*),align=left ]
\item \label{enu:sharpTop}If $\mathcal{I}=\Rtil_{>0}^{*}$ then $\tau_{\mathcal{I}}$
is the sharp topology. Among the balls $B_{r}(x)$ in this topology
we can also have cases where both $r\in\mathcal{I}$ and $x\in\Rtil^{n}$
are not near standard. 
\item \label{enu:FermatTop}If $\mathcal{I}=\R_{>0}$ then $\tau_{\mathcal{I}}$
is called the Fermat topology on $\Rtil^{n}$. Among the balls $B_{r}(x)$
in this topology we can only have cases where the radius is a standard
positive real number. For this reason, open sets in this topology
are also called \emph{large open sets}. Let us note that if $S\subseteq\R^{\bullet}$,
i.e.\ if $S$ consists of near-standard points only (see Section
\ref{sec:BasicNotions}), then the trace of the Fermat topology on
$S$ is induced by the Fermat pseudometric $d_{\text{F}}(x,y)=|\st{x}-\st{y}|$.
In fact, if $B_{r}^{\text{F}}(x)=\{y\in\R^{\bullet}\mid|\st{y}-\st{x}|<r\}$
are the balls in this metric, then $B_{r/2}(x)\cap S\subseteq B_{r}^{\text{F}}(x)\cap S\subseteq B_{r}(x)\cap S$. This justifies the name Fermat topology (introduced in \cite{GK2})
for the topology of large open sets. 
\item \label{enu:setOfRadii-I_a}Let $a\in\R_{>0}$, and set $\mathcal{I}_{a}:=\{[r\cdot\eps^{b}]\in\Rtil\mid r\in\R_{>0}\ ,\ 0<b<a\}$,
then $\mathcal{I}_{a}$ is a set of radii that generates a topology
(on $\Rtil^{n}$) strictly coarser than the sharp and strictly finer
than the Fermat ones. Open sets defined by $\mathcal{I}_{a}$ cannot
contain neighborhoods of radius smaller than $[\eps^{a}]$. 
\item Let $H\subseteq\Rtil_{>0}^{*}$ be a non empty set of positive and
invertible CGN, then $\mathcal{J}_{H}:=\left\{ \bigwedge_{i=1}^{n}r_{i}\cdot h_{i}\mid n\in\N_{>0}\ ,\ r\in\R_{>0}^{n}\ ,\ h\in H^{n}\right\} $
is the smallest set of radii containing $H$. In particular, if $H=\{h\}$
and $h\approx0$ then $\mathcal{J}_{\{h\}}=\left\{ r\cdot h\mid r\in\R_{>0}\right\} $
generates a topology strictly finer than the Fermat one and strictly
coarser than the sharp one. On the contrary, if $h$ is infinite,
i.e.\ $|h|>s$ for all $s\in\R_{>0}$, then it generates a topology
strictly coarser than the Fermat one. Finally, $\mathcal{J}_{\{[\eps^{a}]\}}$
generates a topology strictly finer than the topology generated by
the set of radii $\mathcal{I}_{a}$ described in \ref{enu:setOfRadii-I_a}. 
\end{enumerate}
\end{example}
In the present work, we will only develop examples \ref{enu:sharpTop}
and \ref{enu:FermatTop}.

Any topology on $\Rtil^{n}$ can be used to introduce an equivalence
relation on $\Rtil^{n}$ which permits to define a corresponding class
of strongly internal sets:
\begin{defn}
\label{def:identifiedByTau}Let $\tau$ be a topology on $\Rtil^{n}$,
and $x,y\in\Rtil^{n}$, then we say that $x$, $y$ \emph{are identified
by $\tau$}, and we write $x\asymp_{\tau}y$ if for all $\tau$-open
sets $U\in\tau$ 
\[
x\in U\iff y\in U.
\]

\end{defn}
Clearly, $\asymp_{\tau}$ is an equivalence relation on $\Rtil^{n}$.
\begin{example}
\label{exa:identifiedRelForSharpAndFermat}\ 
\begin{enumerate}[leftmargin=*,label=(\roman*),align=left ]
\item \label{enu:sharpIdentified}If $\tau$ is the sharp topology, then
$x\asymp_{\tau}y$ if and only if $x=y$. 
\item \label{enu:FermatIndentified}If $\tau$ is the Fermat topology, then
$x\asymp_{\tau}y$ if and only if $x\approx y$. 
\end{enumerate}
\end{example}
The following notion concerns membership for $\eps$-dependent objects;
it assures that the class of nets we will consider is always closed
under choosing different representatives with respect to $\asymp_{\tau}$.
\begin{defn}
\label{def:strongMembership} Let $(A_{\eps})$ be a net of subsets
of $\R^{n}$. Moreover, let $(x_{\eps})$ be a net of points in $\R_{M}^{n}$,
then we say that $(x_{\eps})$ \emph{$\tau$-strongly belongs to}
$(A_{\eps})$ and we write 
\[
x_{\eps}\in_{\tau}A_{\eps}
\]
 if 
\begin{enumerate}[leftmargin=*,label=(\roman*),align=left ]
\item $x_{\eps}\in A_{\eps}$ for $\eps$ sufficiently small; 
\item If $[x'_{\eps}]\asymp_{\tau}[x_{\eps}]$, then also $x'_{\eps}\in A_{\eps}$
for $\eps$ sufficiently small. 
\end{enumerate}

Therefore, we can consider the set 
\begin{equation}
\tauInt{A_{\eps}}:=\{[x_{\eps}]\in\Rtil^{n}\mid x_{\eps}\in_{\tau}A_{\eps}\}\label{strongInternal}
\end{equation}
 which, generally speaking, is a subset of the corresponding internal
set 
\begin{equation}
[A_{\eps}]=\left\{ [x_{\eps}]\in\Rtil^{n}\mid x_{\eps}\in A_{\eps}\text{ for }\eps\text{ small}\right\}, \label{eq:internalSetVMO}
\end{equation}
 as defined in \cite{Ob-Ve,Ver11}, because of our definition of strong
membership. Subsets of $\Rtil^{n}$ of the form \eqref{strongInternal}
will be called \emph{$\tau$-strongly internal}. In particular we
simply use the name \emph{strongly internal set} for the case where $\tau$
is the sharp topology and \emph{large internal set} for the case where
$\tau$ is the Fermat topology. In the first one, we use the notations
$\in_{\eps}$ and $\sint{A_{\eps}}$; in the second one we use $\in_{\text{F}}$
and $\Fint{A_{\eps}}$.

\end{defn}
\begin{rem}
\label{rem:largeOpen}\ 
\begin{enumerate}[leftmargin=*,label=(\roman*),align=left ]
\item $\Rtil=\sint{(-e^{\frac{1}{\eps}},e^{\frac{1}{\eps}})}$ is strongly
internal. 
\item \label{enu:otilcIsLargeOpen-1}If $\Omega\subseteq\R^{n}$ is any
open set, then $\widetilde{\Omega}_{c}$ is a large open set. In fact,
if $x\in\widetilde{\Omega}_{c}$, we have $x_{\eps}\in K\comp\Omega$
for $\eps$ small. Since $K$ is compact, $d(K,\R^{n}\setminus\Omega)>0$.
Taking $r\in\R_{>0}$ strictly less than this distance, any $[y_{\eps}]\in B_{r}(x)$
is compactly supported as well. 
\item It is easy to prove that $\sint{A_{\eps}\cap B_{\eps}}=\sint{A_{\eps}}\cap\sint{B_{\eps}}$,
whereas the corresponding property for internal sets is false in general. 
\item Let $\mathcal{P}(-)$ be a property of generalized points in $\Rtil^{n}$ (i.e.\ $\mathcal{P}(-)$ has to be thought of as a syntactical object), 
and set $\ulcorner\mathcal{P}\urcorner:=\{x\in\Rtil^{n}\mid\mathcal{P}(x)\}$ (i.e.\ $\ulcorner\mathcal{P}\urcorner$ is the set-theoretical interpretation of $\mathcal{P}$).
For $x$, $y\in\Rtil^{n}$, we have that $x\asymp_{\tau}y$ if and
only if for each property $\mathcal{P}$, if $\ulcorner\mathcal{P}\urcorner\in\tau$
then $\mathcal{P}(x)$ holds if and only if $\mathcal{P}(y)$ holds.
We can say that $x$ and $y$ are identified by $\tau$ if and only
if these generalized points have the same properties $\mathcal{P}(-)$
which can be interpreted as open sets in the topology $\tau$, i.e.\ such
that the set-theoretical interpretation $\ulcorner\mathcal{P}\urcorner$ of the property $\mathcal{P}$ is an open set in $\tau$. Moreover, if $\mathcal{P}$
is one of these properties and $\mathcal{P}(x)$ holds, then we can
say it is a $\tau$-stable property, i.e.\ also $\mathcal{P}(y)$
holds for $y$ sufficiently near to $x$ with respect to $\tau,$
i.e.\ if $y\asymp_{\tau}x$.
\end{enumerate}
\end{rem}
The following result provides a certain geometrical intuition about
this notion of $\tau$-strong membership and justifies its name. It
also underscores the differences with internal sets as studied in
\cite{Ob-Ve,Ver11}.
\begin{thm}
\label{prop:strongMembershipAndDistanceComplement}
Let $(A_{\eps})$ be a net of subsets of $\R^{n}$ indexed for $\eps\in I$, and let
$(x_{\eps})\in\R_{M}^{n}$. Then the following properties hold:
\begin{enumerate}[leftmargin=*,label=(\roman*),align=left ]
\item $x_{\eps}\in_{\eps}A_{\eps}$ if and only if there exists some $q\in\R_{>0}$
such that $d(x_{\eps},A_{\eps}^{c})>\eps^{q}$ for $\eps$ sufficiently
small, where $A_{\eps}^{c}:=\R^{n}\setminus A_{\eps}$. Therefore,
$x_{\eps}\in_{\eps}A_{\eps}$ if and only if $[d(x_{\eps},A_{\eps}^{c})]$
is invertible in $\Rtil$.
\item $x_{\eps}\in_{\text{F}}A_{\eps}$ if and only if there exists some
$r\in\R_{>0}$ such that $d(x_{\eps},A_{\eps}^{c})>r$ for $\eps$
sufficiently small. 
\end{enumerate}
\end{thm}
\begin{proof}
We proceed for the sharp topology, since the case of the Fermat one
can be treated analogously. Let $x_{\eps}\in_{\eps}A_{\eps}$ and
suppose to the contrary that there exists a sequence $\eps_{k}\searrow0$
such that $d(x_{\eps_{k}},A_{\eps_{k}}^{c})\le\eps_{k}^{k}$ for all
$k\in\N$. For each $k$, pick some $x_{k}'\in A_{\eps_{k}}^{c}$
with $|x_{k}'-x_{\eps_{k}}|<2\eps_{k}^{k}$ and choose $(x'_{\eps})\sim(x_{\eps})$
such that $x'_{\eps_{k}}=x'_{k}$ for all $k$. Then $x_{\eps_{k}}'\not\in A_{\eps_{k}}$
for all $k$, contradicting $x_{\eps}\in_{\eps}A_{\eps}$. Conversely,
let $d(x_{\eps},A_{\eps}^{c})>\eps^{q}$ for $\eps$ small. Then in
particular, $x_{\eps}\in A_{\eps}$. Also, if $(x_{\eps}')\sim(x_{\eps})$
then $d(x_{\eps}',A_{\eps}^{c})>(1/2)\eps^{q}$ for $\eps$ small,
so $x_{\eps}'\in A_{\eps}$. Thus, $x_{\eps}\in_{\eps}A_{\eps}$.
\end{proof}
Hence for $x=[x_{\eps}]\in\Rtil^{n}$, $x_{\eps}\in_{\eps}A_{\eps}$
if and only if $[x_{\eps}]$ is in the interior of $\sint{A_{\eps}}$
with respect to the sharp topology
\begin{cor}
\label{cor:stronglyInternalSharplyOpen} $\sint{A_{\eps}}=\sint{\mathring{A_{\eps}}}$
is open in the sharp topology.
\end{cor}
Note that even in the simplest case of a constant net $\Omega=\Omega_{\eps}$,
the corresponding strongly internal set $\sint{\Omega}$ is contained
in $\widetilde{\Omega}$, but equality in general does not hold (see
\ref{enu:strongInternalAndInterior2} in Example \ref{exp:strongInt}).
\begin{example}
\label{exp:strongInt}\ 
\begin{enumerate}[leftmargin=*,label=(\roman*),align=left ]
\item \label{enu:strongInternalAndInterior} Let $(a_{\eps}),(b_{\eps})\in\R_{M}$,
with $a_{\eps}<b_{\eps}$, then from Prop.\ \ref{prop:strongMembershipAndDistanceComplement}
it is easy to prove that $\sint{[a_{\eps},b_{\eps}]}=\sint{(a_{\eps},b_{\eps})}=\{x\in(a,b)\mid x-a,b-x\in\Rtil^{*}_{>0}\}$,
where we recall that $\Rtil^{*}:=\{x\in\Rtil\mid x\text{ is invertible}\}$.%

\item \label{enu:strongInternalAndInterior2}We always have that $\sint{A_{\eps}}\subseteq\text{int}_{\text{s}}([A_{\eps}])$,
where $[A_{\eps}]$ is the internal set generated by the net $(A_{\eps})_{\eps}$
in the sense of \cite{Ob-Ve,Ver11} (see \eqref{eq:internalSetVMO})
and $\text{int}_{\text{s}}(B)$ is the interior of $B\subseteq\Rtil^{n}$
in the sharp topology. Indeed, $\sint{A_{\eps}}\subseteq[A_{\eps}]$
by definition, and $\sint{A_{\eps}}$ is open in the sharp topology
by Corollary \ref{cor:stronglyInternalSharplyOpen}. However, the
reverse inclusion is false: let $A_{\eps}=\R\setminus\{0\}$. Then
$\sint{A_{\eps}}=\Rtil^{*}\subsetneq\text{int}_{\text{s}}([A_{\eps}])=\text{int}_{\text{s}}(\Rtil)=\Rtil$.%
\end{enumerate}
\end{example}
\noindent We close this section with the following result, which provides
a certain intuition on the net of open sets $\Omega_{\eps}\subseteq\R^{n}$
that generates the strongly internal set $\sint{\Omega_{\eps}}$.
We recall that a net $(B_{\eps})$ of subsets of $\R^{n}$ is called
\emph{sharply bounded} if there exists $N\in\R_{>0}$ such that 
\[
\forall^0\eps\,\forall a\in B_{\eps}:\ |a|\le\eps^{-N}.
\]

\begin{thm}
\noindent \label{prop:OmegaEpsAroundBoundedAndCptSupp}Let $(\Omega_{\eps})$
be a net of subsets in $\R^{n}$ for all $\eps$, and $(B_{\eps})$ a sharply
bounded net such that $[B_{\eps}]\subseteq\sint{\Omega_{\eps}}$,
then 
\[
\forall^0\eps:\ B_{\eps}\subseteq\Omega_{\eps}.
\]
\end{thm}
\begin{proof}
\noindent By contradiction assume that we can find sequences $(\eps_{k})_{k}$
and $(x_{k})_{k}$ such that $\eps_{k}\downarrow0$ and $x_{k}\in B_{\eps_{k}}\setminus\Omega_{\eps_{k}}$.
Defining $x_{\eps}:=x_{k}$ for $\eps\in(\eps_{k+1},\eps_{k}]$, and $x_\eps\in B_\eps$ otherwise, we
have that $x:=[x_{\eps}]$ is moderate since $(B_{\eps})$ is sharply
bounded. Hence $x\in[B_{\eps}]$, but $x_{\eps_{k}}\notin\Omega_{\eps_{k}}$
by construction, hence $x\notin\sint{\Omega_{\eps}}$ by Def. \ref{def:strongMembership},
which is impossible because $[B_{\eps}]\subseteq\sint{\Omega_{\eps}}$.\end{proof}
\begin{example}
\noindent \label{exa:otilcNotInternal}Let $\Omega\subseteq\R^{n}$
be open and bounded. Then $\widetilde{\Omega}_{c}$ is not strongly
internal. Indeed, suppose that $\widetilde{\Omega}_{c}=\sint{\Omega_{\eps}}$.
Let $(K_{n})_{n}$ be a compact exhaustion of $\Omega$. Then by the
previous theorem, there exist $\eps_{n}$ such that $K_{n}\subseteq\Omega_{\eps}$,
for each $\eps\le\eps_{n}$. W.l.o.g., $(\eps_{n})_{n}$ decreasingly
tends to $0$ and $\eps_{n}\le d_{n}:=d(K_{n-1},\co K_{n})$. Choose
$x_{\eps}\in K_{n-1}\setminus K_{n-2}$ for each $\eps_{n+1}<\eps\le\eps_{n}$.
As $\Omega$ is bounded, $(x_{\eps})_{\eps}$ is moderate. Then $[x_{\eps}]\notin\widetilde{K_{n}}$,
for each $n\in\N$. On the other hand, $[x_{\eps}]\in\sint{\Omega_{\eps}}$,
since $d(x_{\eps},\co\Omega_{\eps})\ge\eps$. For, if $|y_{\eps}-x_{\eps}|\le\eps$
and $\eps_{n+1}<\eps\le\eps_{n}$, then $|y_{\eps}-x_{\eps}|\le d_{n}$
and $x_{\eps}\in K_{n-1}$. Hence $y_{\eps}\in K_{n}\subseteq\Omega_{\eps}$. 
\end{example}
\noindent The following theorem sheds some light on the relationship
between internal sets and strongly internal sets, implying e.g.\ that
they generate the same $\sigma$-algebra:
\begin{thm}
\noindent Let $A_{\eps}\subseteq\R^{n}$. Then we have:
\begin{enumerate}[leftmargin=*,label=(\roman*),align=left ]
\item \label{enu:1sigma}$\sint{A_{\eps}}=\bigcup_{m\in\N}[A_{m,\eps}]$,
where $A_{m,\eps}=\{x\in\R^{n}:d(x,A_{\eps}^{c})\ge\eps^{m}\}$. 
\item \noindent \label{enu:2sigma}$[A_{\eps}]=\bigcap_{m\in\N}\sint{A_{m,\eps}}$,
where $A_{m,\eps}=\{x\in\R^{n}:d(x,A_{\eps})<\eps^{m}\}$. 
\end{enumerate}
\end{thm}
\begin{proof}
\noindent \ref{enu:1sigma} Let $[x_{\eps}]\in\Rtil^{n}$. By Theorem
\ref{prop:strongMembershipAndDistanceComplement}, $[x_{\eps}]\in\sint{A_{\eps}}$
iff $d(x_{\eps},A_{\eps}^{c})\ge\eps^{m}$ for small $\eps$, for
some $m\in\N$.

\ref{enu:2sigma} Let $[x_{\eps}]\in\Rtil^{n}$. By \cite[Prop. 2.1]{Ob-Ve},
$[x_{\eps}]\in[A_{\eps}]$ iff $d(x_{\eps},A_{\eps})\sim0$ iff $d(x_{\eps},A_{\eps})<\eps^{m}$
for each $m\in\N$ (and for each representative $(x_{\eps})$).\end{proof}
\begin{thm}
Let $(A_{\eps})$ be sharply bounded. Then 
\[
\sint{A_{\eps}}\subseteq\sint{B_{\eps}}\iff\sup_{x\in B_{\eps}^{c}}d(x,A_{\eps}^{c})\sim0.
\]
\end{thm}
\begin{proof}
$\Rightarrow$: by the previous theorems, we have $[A_{m,\eps}]\subseteq\sint{B_{\eps}}$,
and thus $A_{m,\eps}\subseteq B_{\eps}$, i.e., $B_{\eps}^{c}\subseteq A_{m,\eps}^{c}$
for small $\eps$, for each $m$, where $A_{m,\eps}=\{x\in\R^{n}:d(x,A_{\eps}^{c})\ge\eps^{m}\}$.

\noindent $\Leftarrow$: if $x\notin\sint{B_{\eps}}$, then there
exists a representative $(x_{\eps})$ and $\eps_{n}\to0$ such that
$x_{\eps_{n}}\in B_{\eps_{n}}^{c}$ for each $n$. It is given that
there exist $x'_{\eps_{n}}\in A_{\eps_{n}}^{c}$ such that $|x_{\eps_{n}}-x'_{\eps_{n}}|\le\nu_{\eps_{n}}$
with $\nu_{\eps}\sim0$. Let $x'_{\eps}:=x_{\eps}$ if not yet defined.
Then $x=[x'_{\eps}]\notin\sint{A_{\eps}}$.\end{proof}
\begin{cor}
\noindent If $(A_{\eps})$ and $(B_{\eps})$ are sharply bounded nets,
then $\sint{A_{\eps}}=\sint{B_{\eps}}$ if and only if the Hausdorff
distance $d_{H}(A_{\eps}^{c},B_{\eps}^{c})\sim0$.\end{cor}
\begin{defn}
\noindent $A\subseteq\Rtil^{n}$ is convex if for each $x,y\in A$
and $t\in\widetilde{[0,1]}$, $tx+(1-t)y\in A$.\end{defn}
\begin{lem}
\noindent \label{lem:convex} If $A\subseteq\Rtil^{n}$ is convex,
internal and sharply bounded, then $A$ has a representative consisting
of convex sets.\end{lem}
\begin{proof}
\noindent By \cite{Ob-Ve}, $A=[A_{\eps}]$ for some sharply bounded
net $(A_{\eps})_{\eps}$. We show that $A=[\conv(A_{\eps})]$, where
$\conv(X)$ denotes the convex closure of $X$. Let $x\in[\conv(A_{\eps})]$.
Then $x_{\eps}\in\conv(A_{\eps})$ for sufficiently small $\eps$
and for some representative $(x_{\eps})$ of $x$. By Carath{\'e}odory's
theorem in convex geometry, there exist $a_{0,\eps}$, \dots, $a_{n,\eps}\in A_{\eps}$
such that $x_{\eps}\in\conv\{a_{0,\eps},\dots,a_{n,\eps}\}$, i.e.,
there exist $\lambda_{0,\eps},\dots,\lambda_{n,\eps}\in[0,1]_{\R}$
such that $x_{\eps}=\sum_{j=0}^{n}\lambda_{j,\eps}a_{j,\eps}$. Since
$(A_{\eps})_{\eps}$ is sharply bounded, $[a_{j,\eps}]=a_{j}$ for
some $a_{j}\in A\subseteq\Rtil^{n}$. Hence $x=\sum_{j=0}^{n}\lambda_{j}a_{j}\in A$
(with $\lambda_{j}=[\lambda_{j,\eps}]\in\widetilde{[0,1]}$ and $\sum_{j=0}^{n}\lambda_{j}=1$),
since $A\subseteq\Rtil^{n}$ is assumed to be convex.
\end{proof}

\section{\label{sec:locLipFun}Locally Lipschitz functions}
\begin{defn}
\label{def:locLip-1}Let $U\subseteq\Rtil^{n}$. Then $f:U\to\Rtil^{m}$
is called \emph{Lipschitz} if there exists some $L\in\Rtil$ such
that $|f(x)-f(y)|\le L|x-y|$ for all $x$, $y\in U$, where $|\ |$
is the natural extension of the Euclidean norm to generalized points,
i.e.\ $|[x_{\eps}]|:=[|x_{\eps}|]$. The function $f$ is called
\emph{locally Lipschitz} with respect to some topology $\tau$ on
$U$ if every $x\in U$ possesses a $\tau$-neighborhood in which
$f$ is Lipschitz in this sense.
\end{defn}
It is immediate from the definition that a map $f:U\to\Rtil^{m}$
is Lipschitz if and only if $\exists N\in\N:\ |f(x)-f(y)|\le[\eps^{-N}]|x-y|$
for all $x$, $y\in U$. Moreover, any locally Lipschitz function
in the sharp topology (in particular any locally Lipschitz function
in the Fermat topology) is also continuous in the sharp topology.
\begin{example}
\label{exa:u-temperedIsLipOnLargeBalls} As we will see in the next
section, any map $u$: $\widetilde{\Omega}_{c}\freccia\Rtil$ generated
by a CGF or any map $u$: ${\Rtil^{n}}\freccia\Rtil$ generated by
a tempered generalized function is locally Lipschitz for the Fermat
topology.
\end{example}
The following result shows, in particular, that the composition of
locally Lipschitz maps in the sharp topology is again locally Lipschitz,
and gives sufficient conditions for the corresponding property in
the Fermat topology.
\begin{lem}
\label{lem:propLip}Let $A\subseteq\Rtil^{a}$, $B\subseteq\Rtil^{b}$,
$C\subseteq\Rtil^{c}$, $D\subseteq\Rtil^{d}$ and $f:A\freccia B$
and $g:C\freccia D$ be locally Lipschitz maps in the topology $\tau$.
Then, if $\tau$ is the Fermat topology, we have 
\begin{enumerate}[leftmargin=*,label=(\roman*),align=left ]
\item \label{enu:locLipLargeImpliesContLargeNoInf}If $f$ is locally Lipschitz
with respect to \emph{finite} Lipschitz constants, then it is also
continuous in the Fermat topology. 
\item \label{enu:internalCont}If $B=C$ and $f$ is continuous in the Fermat
topology, then $g\circ f$ is locally Lipschitz. 
\item \label{enu:externalGlobLip}If $B=C$ and $g$ is Lipschitz, then
$g\circ f$ is locally Lipschitz. 
\end{enumerate}

\noindent Whereas if $\tau$ is the sharp topology and $B=C$, then
$g\circ f$ is locally Lipschitz in the same topology.

\end{lem}
\begin{proof}
\ref{enu:locLipLargeImpliesContLargeNoInf}: one easily sees that
$f$ is continuous for the Fermat topology iff 
\[
\forall x\in A\,\forall\eps\in\R_{>0}\,\exists\delta\in\R_{>0}\,\forall y\in A:\ (|x-y|\le\delta\implies|f(x)-f(y)|\le\eps),
\]
 which is clearly satisfied if $f$ satisfies the given conditions
of \ref{enu:locLipLargeImpliesContLargeNoInf}.

The proofs of the other parts are formally equal to the standard ones
in metric spaces.\end{proof}
\begin{rem}
\label{rem:largeOpenSets} We emphasize that our notion of Lipschitz
map differs from the classical definition in a metric space, e.g.\ with
respect to the sharp metrics on $\Rtil^{n}$, $\Rtil^{m}$, because
both the Lipschitz constant $L$ and the generalized metric $|x-y|$
assume values in $\Rtil$. On the other hand, it is the natural generalization
of the classical notion to the non-Archimedean ring $\Rtil$. In fact,
if $U\subseteq\R^{n}$ and $f:U\longrightarrow\R^{m}$ is Lipschitz
in the usual sense, then it is also Lipschitz in the sense of Def.\ \ref{def:locLip-1}.
Moreover, if this $f:U\longrightarrow\R^{m}$ is locally Lipschitz
in the usual sense with respect to the Euclidean topology, then viewing
$f$ as a CGF (i.e.\ through the embedding $\mathcal{C}^{0}(U,\R^{m})\subseteq\mathcal{D}'(U)^{m}\subseteq\gs(U)^{m}$),
it is easy to prove that the induced map $f:\widetilde{U}_{c}\longrightarrow\Rtil^{m}$
(which extends the original $f$) is locally Lipschitz with respect
to the Fermat topology with finite Lipschitz constant.
\end{rem}
While clearly on $\R$ (as in any metric space) from a local Lipschitz
condition it is possible to obtain a global one on compact sets, this
is not directly translatable into $\Rtil$ with the above concept
of locally Lipschitz maps. In fact, this property already fails on
finite sets. E.g., let $U=\{0,e_{S}\}\subseteq\Rtil$, where $e_{S}\ne0$ is a zero divisor,
and let $f(0):=0$ and $f(e_{S}):=1$. Then $f$ is locally Lipschitz
for the Fermat topology, but not globally Lipschitz, since $1=|f(e_{S})-f(0)|\le C|e_{S}-0|=Ce_{S}$
does not hold for any $C\in\Rtil$. 
We still have the following: 
\begin{defn}
Let $U\subseteq\Rtil^{n}$. We call $f$: $U\to\Rtil^{m}$ \emph{pointwise
Lipschitz} if for each $x,y\in U$, there exists some $C\in\Rtil$
such that $|f(y)-f(x)|\le C|y-x|$.\\
 We call $f$ \emph{strongly locally Lipschitz} w.r.t.\ the topology
$\tau$ if every $x,y\in U$ possess $\tau$-neighbourhoods $V_{x}$
and $V_{y}$ respectively such that $f$ is Lipschitz on $V_{x}\cup V_{y}$.\end{defn}
\begin{thm}
\label{thm:cptLip}Let $\tau$ be either the sharp topology or the
Fermat topology. Let $K\subseteq\Rtil^{n}$ be $\tau$-compact. 
\begin{enumerate}[leftmargin=*,label=(\roman*),align=left ]
\item \label{enu:stronglyLocLipToGlobLip} If $f$: $K\to\Rtil^{m}$ is
$\tau$-strongly locally Lipschitz on $K$, then $f$ is globally
Lipschitz on $K$. 
\item \label{enu:locLipToGlobLip} Let $f$: $K\to\Rtil^{m}$ be a $\tau$-locally
Lipschitz and pointwise Lipschitz map. Let for each $x,y\in K$ with
$x\not\asymp_{\tau}y$ necessarily $|x-y|\ge[\eps^{m}]$ for some
$m\in\N$ (if $\tau$ is the sharp topology), resp.\ $|x-y|\ge r$
for some $r\in\R_{>0}$ (if $\tau$ is the Fermat topology). Then
$f$ is globally Lipschitz on $K$.
\end{enumerate}
\end{thm}
\begin{proof}
\ref{enu:stronglyLocLipToGlobLip} For each $n\in\N$, call $A_{n}$
the $\tau$-interior of the set $\{(x,y)\in K\times K:|f(y)-f(x)|\le[\eps^{-n}]|y-x|\}$.
Since $f$ is strongly locally Lipschitz on $K$, every $(x,y)\in K\times K$
belongs to $A_{n}$ for some $n\in\N$. In fact, $(A_{n})_{n\in\N}$
is a $\tau$-open cover of $K\times K$. As $K\times K$ is $\tau$-compact,
it follows that $K\times K\subseteq A_{N}$ for some $N\in\N$. Hence
$f$ is Lipschitz on $K$.

\ref{enu:locLipToGlobLip}: by \ref{enu:stronglyLocLipToGlobLip},
we only have to show that $f$ is $\tau$-strongly locally Lipschitz
on $K$. Thus consider any $x,y\in K$. Choose a $\tau$-neighbourhood
$V_{x}$ of $x$ (resp.\ $V_{y}$ of $y$) on which $f$ is Lipschitz.
If $x=y$, then $V_{x}$ is also a $\tau$-neighbourhood of $y$,
and thus $f$ is trivially Lipschitz on $V_{x}=V_{x}\cup V_{x}$.
Otherwise, $|x-y|\ge[\eps^{m}]$ for some $m\in\N$ by assumption.
By shrinking $V_{x}$ and $V_{y}$, we may assume that $V_{x}\subseteq B_{[\eps^{m}]/3}(x)$
and $V_{y}\subseteq B_{[\eps^{m}]/3}(y)$. Then there exists $N\in\N$
such that for any $\xi\in V_{x}$ and $\eta\in V_{y}$ 
\begin{align*}
|f(\eta)-f(\xi)| & \le|f(\eta)-f(y)|+|f(y)-f(x)|+|f(x)-f(\xi)|\\
 & \le[\eps^{-N}]|\eta-y|+[\eps^{-N}]|y-x|+[\eps^{-N}]|x-\xi|
\end{align*}
 since $f$ is Lipschitz on $V_{x}$ and on $V_{y}$ and pointwise
Lipschitz. Now $|x-\xi|\le[\eps^{m}]/3\le|x-y|/3$ and $|y-\eta|\le|x-y|/3$.
Hence 
\[
|\eta-\xi|\ge|x-y|-|x-\xi|-|y-\eta|\ge|x-y|/3
\]
 and thus $|f(\eta)-f(\xi)|\le5[\eps^{-N}]|\eta-\xi|$. For the Fermat
topology one proceeds similarly, using a suitable $r\in\R_{>0}$ instead
of $[\eps^{m}]$.\end{proof}
\begin{example}
\label{exa:i-isNotLocLipwrtFermat}The function $i(x):=1$ if $x\approx0$
and $i(x):=0$ otherwise is globally Lipschitz with constant 1 with
respect to the $|-|_{e}$ norm, but it is not locally Lipschitz with
respect to the Fermat topology in the sense of Def.\ \ref{def:locLip-1}.
In fact, if $x\approx0$ and $y\not{\!\!\approx}\,0$, then $|i(x)-i(y)|_{e}=1$
and the pseudo-valuation $v(x-y)\le0$, otherwise $y$ would be infinitesimal.
Therefore $|x-y|_{e}=e^{-v(x-y)}\ge1=|i(x)-i(y)|_{e}$. On the other
hand, $i$ is locally Lipschitz in the sharp topology in the sense
of Def.\ \ref{def:locLip-1}; indeed, any point can be enclosed in
an infinitesimal ball, where the function $i$ is constant. However,
$i$ is not locally Lipschitz in the Fermat topology. Assume that
it verifies 
\begin{equation}
|i(0)-i(x)|\le L\cdot|x|\quad\forall x\in B_{r}(0),\label{eq:abs-i-LipAround-0}
\end{equation}
 where $r\in\R_{>0}$. It suffices to take as $x$ any oscillating
number with $|x|\le r$ but with $x_{\eps_{k}}=0$ for some sequence
$(\eps_{k})_{k}\downarrow0$ and $x_{\eta_{k}}=r/2$ along another
sequence to get that $x\not{\!\!\approx}\,0$ but $L_{\eps_{k}}\cdot|x_{\eps_{k}}|=0$.
Finally, let us note that taking e.g.\ $x=\frac{1}{n}$ in \eqref{eq:abs-i-LipAround-0}
we necessarily would have that $L$ is infinite, as our intuition
about the function $i$ would suggest. We recall that the map $i$
is smooth in the sense of \cite{AFJ}, it is continuous in the sharp
topology and its derivative, in the sense of \cite{AFJ}, vanishes
everywhere.
\end{example}
Unfortunately, a large number of sets in which one is interested are
not compact for the sharp or Fermat topology. For a start, no infinite
subset $U\subseteq\R^{n}$ is compact w.r.t.\ to the sharp topology,
since its relative topology on $U$ is the discrete topology. But
also internal and strongly internal sets are almost never compact,
as the following theorem shows. We recall that $U\subseteq\Rtil^{n}$
is closed under finite interleaving if for each $x,y\in U$ and $S\subseteq I$
also $e_{S}x+e_{\co S}y\in U$. Any internal set and any strongly
internal set is closed under finite interleaving.
\begin{thm}
\label{thm:propCmpWRTLarge}Let $\tau$ be either the sharp topology
or the Fermat topology. 
\begin{enumerate}[leftmargin=*,label=(\roman*),align=left ]
\item \label{enu:[0,1]-notCmpSharp}Let $U\subseteq\Rtil^{n}$ be closed
under finite interleaving. If there exist $x,y\in U$ with $x\not\asymp_{\tau}y$,
then $U$ is not $\tau$-compact. 
\item \label{enu:stdCmpAreFermatCmp}Let $K\comp\R^{n}$ and let $U\subseteq K^{\bullet}$
with $K\subseteq\{\st x:x\in U\}$. Then $U$ is compact in the Fermat
topology on $\Rtil^{n}$.
\end{enumerate}
\end{thm}
\begin{proof}
\ref{enu:[0,1]-notCmpSharp}: if $x\ne y$, then there exists $S\subseteq I$
with $e_{S}\ne0$ and $m\in\N$ such that $|x-y|e_{S}\ge[\eps]^{m}e_{S}$.
We can find (e.g.\ by extracting subsequences from $S$) mutually
disjoint $S_{n}\subseteq(0,1]$ such that $S=\bigcup_{n\in\N}S_{n}$
and $e_{S_{n}}\ne0$ for each $n$. Call $z_{n}:=xe_{S_{n}}+ye_{\co S_{n}}\in U$.
If $p\ne n$, then 
\[
|z_{p}-z_{n}|e_{S_{n}}=|y-x|e_{S_{n}}\ge[\eps]^{m}e_{S_{n}}.
\]
 We show that the sharply open cover $\{B_{[\eps]^{m}/3}(x):x\in U\}$
of $U$ has no finite subcover. For, suppose it had, then by the pigeon
hole principle there would exist $n\ne p$ such that $z_{n}$ and
$z_{p}$ belong to the same ball $B_{[\eps]^{m}/3}(x)$, whence $|z_{p}-z_{n}|e_{S_{n}}\le\frac{2}{3}[\eps]^{m}e_{S_{n}}$,
a contradiction. For the Fermat topology one proceeds similarly, using
a cover $\{B_{r}(x):x\in U\}$ with a suitable $r\in\R_{>0}$ instead.

\noindent \ref{enu:stdCmpAreFermatCmp}: let $(A_{j})_{j\in J}$ be
a cover of $U$ by large open sets. Then for all $x\in U$ we can
find $r_{x}\in\R_{>0}$ and $j_{x}\in J$ such that $B_{r_{x}}(\st x)\subseteq A_{j_{x}}$.
Therefore, the Euclidean balls $\left(B_{r_{x}}^{\text{E}}(\st x)\right)_{x\in U}$
cover $K$ and we extract a finite subcover $\Eball{}_{r_{x_{1}}}(\st{x_{1}}),\ldots,\Eball{}_{r_{x_{n}}}(\st{x_{n}})$.
Hence the balls with the same radius but taken with respect to the
generalized absolute value $B_{r_{x_{1}}}(\st{x_{1}}),\ldots,B_{r_{x_{n}}}(\st{x_{n}})$
and the corresponding $A_{j_{x_{1}}},\ldots,A_{j_{x_{n}}}$ cover
$K^{\bullet}$ (and thereby $U$).
\end{proof}
\noindent In the next section, we will show that these restrictions
can be overcome if one restricts to certain maps $f$ with `internal
structure'.

\section{\label{sec:ColAlgSubset}The Colombeau algebra on a subset of $\Rtil^{d}$}

\noindent In this section we shall introduce a set of maps which are
locally Lipschitz in the sharp topology and includes CGF. We will
first introduce the notion of a net $(u_{\eps})$ defining a generalized
smooth map $X\longrightarrow Y$, where $X\subseteq\Rtil^{n}$, $Y\subseteq\Rtil^{d}$.
This is a net of smooth functions $u_{\eps}\in\cinfty(\Omega_{\eps},\R^{d})$
which induces well defined maps of the form $[\partial^{\alpha}u_{\eps}(-)]:\sint{\Omega_{\eps}}\freccia\R^{d}$,
for every multi-index $\alpha$.
\begin{defn}
\noindent \label{def:netDefMap}Let $X\subseteq\Rtil^{n}$ and $Y\subseteq\Rtil^{d}$
be subsets of generalized points. Let $(\Omega_{\eps})$ be
a net of open sets of $\R^{n}$, and $(u_{\eps})$ be a net of smooth
functions, with $u_{\eps}\in\cinfty(\Omega_{\eps},\R^{d})$. Then
we say that 
\[
(u_{\eps})\textit{ defines a generalized smooth map }X\longrightarrow Y
\]
 if: 
\begin{enumerate}[leftmargin=*,label=(\roman*),align=left ]
\item \label{enu:dom-cod}$X\subseteq\sint{\Omega_{\eps}}$ and $[u_{\eps}(x_{\eps})]\in Y$
for all $x=[x_{\eps}]\in X$ 
\item \label{enu:partial-u-moderate}$\forall[x_{\eps}]\in X\,\forall\alpha\in\N^{n}:\ (\partial^{\alpha}u_{\eps}(x_{\eps}))\in\R_{M}^{d}$.
\end{enumerate}

\noindent The notation 
\[
\forall[x_{\eps}]\in X:\ \mathcal{P}\{(x_{\eps})\}
\]
 means 
\[
\forall(x_{\eps})\in\left(\Omega_{\eps}\right)_{M}:\ [x_{\eps}]\in X\then\mathcal{P}\{(x_{\eps})\}
\]
 i.e.\ for all representatives $(x_{\eps})$ of the point $[x_{\eps}]\in X$
the property $\mathcal{P}\{(x_{\eps})\}$ holds.

\end{defn}
\noindent A generalized smooth map is simply a function of the form
$f=[u_{\eps}(-)]|_{X}$:
\begin{defn}
\noindent \label{def:generalizedSmoothMap} Let $X\subseteq\Rtil^{n}$
and $Y\subseteq\Rtil^{d}$, then we say that 
\[
f:X\longrightarrow Y\text{ is a \emph{generalized smooth function} (GSF)}
\]
 if there exists a net $u_{\eps}\in\cinfty(\Omega_{\eps},\R^{d})$
defining $f$ in the sense of Def.\ \ref{def:netDefMap}, such that
$f$ is the map 
\begin{equation}
f=[u_{\eps}(-)]|_{X}:[x_{\eps}]\in X\mapsto[u_{\eps}(x_{\eps})]\in Y.\label{eq:f-u-relations}
\end{equation}

\noindent We will also say that $f$ \emph{is generated (or defined)
by the net of smooth functions} $(u_{\eps})$. The set of all GSF
$X\to Y$ will be denoted by $\Gcinf(X,Y)$.%

\end{defn}
Let us note explicitly that definitions \ref{def:netDefMap} and \ref{def:generalizedSmoothMap}
in fact state minimal logical conditions to obtain a set-theoretical
map defined by a net of smooth functions. In particular, Proposition
\ref{prop:indepRepr} below will show that the equality \eqref{eq:f-u-relations}
is meaningful, i.e.\ that we have independence from the representatives
for all derivatives $[x_{\eps}]\in X\mapsto[\partial^{\alpha}u_{\eps}(x_{\eps})]\in\Rtil^{d}$,
$\alpha\in\N^{n}$.

We first show that we can always find globally defined representatives.
The generalization where the domains of representatives $u_{\eps}$
depend on $\eps$ (see \cite{SV06} for a recent survey concerning
applications of this generalization) can thus be avoided since it
does not lead to a larger class of generalized functions. We will
use it also to compare CGF and GSF.
\begin{lem}
\label{lem:fromOmega_epsToRn}Let $X\subseteq\Rtil^{n}$ and $Y\subseteq\Rtil^{d}$.
Then $f:X\longrightarrow Y$ is a GSF if and only if there exists
a net $v_{\eps}\in\cinfty(\R^{n},\R^{d})$ defining a generalized
smooth map $X\longrightarrow Y$ such that $f=[v_{\eps}(-)]|_{X}$.\end{lem}
\begin{proof}
The stated condition is clearly sufficient. Conversely, assume that
$f:X\longrightarrow Y$ is defined by the net $u_{\eps}\in\cinfty(\Omega_{\eps},\R^{d})$.
For every $\eps\in I$ let $\Omega'_{\eps}:=\left\{ x\in\Omega_{\eps}\mid d(x,\Omega_{\eps}^{c})>e^{-\frac{1}{\eps}}\right\} $
and choose $\chi_{\eps}\in\mathcal{C}^{\infty}(\Omega_{\eps})$ with
$\supp(\chi_{\eps})\subseteq\Omega'_{\eps/2}$ and $\chi_{\eps}=1$
in a neighborhood of $\Omega'_{\eps}$. Set $v_{\eps}:=\chi_{\eps}\cdot u_{\eps}$,
so that $v_{\eps}\in\cinfty(\R^{n},\R^{d})$. If $x=[x_{\eps}]\in X\subseteq\sint{\Omega_{\eps}}$,
then $x_{\eps}\in\Omega'_{\eps}$ for small $\eps$ by Th.\ \ref{prop:strongMembershipAndDistanceComplement},
so for all $\alpha\in\N^{n}$ we get $\partial^{\alpha}v_{\eps}(x_{\eps})=\partial^{\alpha}u_{\eps}(x_{\eps})$
for small $\eps$. Therefore, $(v_{\eps})_{\eps}$ defines a GSF $X\longrightarrow Y$
and clearly $f=[u_{\eps}(-)]|_{X}=[v_{\eps}(-)]|_{X}$.
\end{proof}
We also need to prove that for GSF certain moderateness conditions
hold:
\begin{lem}
\label{lemma:ModerateUniform} Let $(A_{n})_{n\in\N}$ be a decreasing
sequence of non-empty, internal, sharply bounded subsets of $\GenR^{d}$.
Let $(u_{\eps})$ be a net of maps $\R^{d}\to\R^{m}$. Then for any
sharply bounded representatives $(A_{n,\eps})_{\eps}$ of $A_{n}$, 
\begin{enumerate}[leftmargin=*,label=(\roman*),align=left ]
\item \label{enu:moderCondFirst}For all $[x_{\eps}]\in\bigcap_{n\in\N}A_{n}$
we have $(u_{\eps}(x_{\eps}))_{\eps}\in\R_{M}^{m}$ if and only if
$\exists N\in\N\,\forall^{0}\eps:\ \sup_{x\in A_{N,\eps}+\eps^{N}}\abs{u_{\eps}(x)}\le\eps^{-N}$. 
\item \label{enu:moderCondSecond}For all $[x_{\eps}]\in\bigcap_{n\in\N}A_{n}$
we have $(u_{\eps}(x_{\eps}))\sim0$ if and only if $\forall m\in\N\,\exists N\in\N\,\forall^{0}\eps:\ \sup_{x\in A_{\eps}+\eps^{N}}\abs{u_{\eps}(x)}\le\eps^{m}$. 
\end{enumerate}
\end{lem}
\begin{proof}
\ref{enu:moderCondFirst} $\Rightarrow$: By \cite[Prop. 2.9]{Ob-Ve},
for each $m\in\N$, there exists $\eta_{m}\in I$ such that for
each $\eps\le\eta_{m}$ and $x\in A_{m,\eps}$, $d(x,A_{k,\eps})\le\eps^{m}$,
for each $k\le m$. W.l.o.g., $(\eta_{m})_{m\in\N}$ decreasingly
tends to $0$. By contraposition, let 
\[
\forall n\in\N\,\forall\eta\in I\,\exists\eps\le\eta:\ \sup_{x\in A_{n,\eps}+\eps^{n}}\abs{u_{\eps}(x)}>\eps^{-n}.
\]
 Then we can find a strictly decreasing sequence $(\eps_{n})_{n\in\N}$
and $x_{\eps_{n}}\in A_{n,\eps_{n}}+\eps_{n}^{n}$ such that $\eps_{n}\le\eta_{n}$
and $\abs{u_{\eps_{n}}(x_{\eps_{n}})}>\eps_{n}^{-n}$, $\forall n\in\N$.
Choose $x_{\eps}\in A_{m,\eps}$, if $\eta_{m+1}<\eps\le\eta_{m}$
and $\eps\notin\{\eps_{n}:n\in\N\}$. Then for each $n\in\N$, $(d(x_{\eps},A_{n,\eps}))_{\eps}\sim0$.
By \cite[Prop. 2.1]{Ob-Ve}, $\tilde{x}:=[x_{\eps}]\in\bigcap_{n\in\N}A_{n}$
($(x_{\eps})_{\eps}$ is moderate, since $(A_{n,\eps})_{\eps}$ are
sharply bounded). Yet $(u_{\eps}(x_{\eps}))_{\eps}\notin\R_{M}^{m}$.\\
 \ref{enu:moderCondFirst} $\Leftarrow$: let $[x_{\eps}]\in\bigcap_{n\in\N}A_{n}$.
Let $N\in\N$ as in the statement. By \cite[Prop. 2.1]{Ob-Ve}, $(d(x_{\eps},A_{n,\eps}))_{\eps}\sim0$
for each $n\in\N$. In particular, $x_{\eps}\in A_{N,\eps}+\eps^{N}$
for small $\eps$. Hence $\abs{u_{\eps}(x_{\eps})}\le\sup_{x\in A_{N,\eps}+\eps^{N}}\abs{u_{\eps}(x)}\le\eps^{-N}$
for small $\eps$.\\
 \ref{enu:moderCondSecond} Similar.\end{proof}
\begin{thm}
\label{prop:indepRepr}Let $X\subseteq\Rtil^{n}$. If $(u_{\eps})$
defines a generalized smooth map $X\to\Rtil^{d}$, then $\forall[x_{\eps}],[x'_{\eps}]\in X:\ [x_{\eps}]=[x'_{\eps}]\ \Rightarrow\ (u_{\eps}(x_{\eps}))\sim(u_{\eps}(x'_{\eps}))$.\end{thm}
\begin{proof}
W.l.o.g.\ $u_{\eps}\in\cinfty(\R^{n},\R^{d})$ by Lemma \ref{lem:fromOmega_epsToRn}.
Let $[x_{\eps}]\in X$. Applying Lemma \ref{lemma:ModerateUniform}
to $A_{n}=\{[x_{\eps}]\}$ and to $\partial_{j}u_{\eps}$ ($j=1,\dots,n$),
we find $N\in\N$ such that for each $(x'_{\eps})\in\R_{M}^{n}$ with
$|x_{\eps}-x'_{\eps}|\le\eps^{N}$ for small $\eps$, 
\begin{equation}
\abs{u_{\eps}(x'_{\eps})-u_{\eps}(x_{\eps})}\le|x'_{\eps}-x_{\eps}|\sup_{\abs{x-x_{\eps}}\le\eps^{N}}\abs{\nabla u_{\eps}(x)}\le\eps^{-N}|x'_{\eps}-x_{\eps}|,\quad\text{for small \ensuremath{\eps}}.\label{eq:GSFLipschitz}
\end{equation}
 Choosing in particular $[x'_{\eps}]=[x_{\eps}]\in X$, then $(x'_{\eps})\sim(x_{\eps})$,
hence also $(u_{\eps}(x_{\eps}))\sim(u_{\eps}(x'_{\eps}))$ by \eqref{eq:GSFLipschitz}.
\end{proof}
Consequently, the GSF $f=[u_{\eps}(-)]$: $X\to Y$ is well-defined
by its representative. We now turn to the derivatives.
\begin{thm}
\label{thm:indepRepreAndLipCond}Let $X\subseteq\Rtil^{n}$. If $f=[u_{\eps}(-)]\in\Gcinf(X,\Rtil^{d})$,
then
\begin{enumerate}[leftmargin=*,label=(\roman*),align=left ]
\item \label{enu:locLipSharp}$f$: $X\to\Rtil^{d}$ is locally Lipschitz
in the sharp topology 
\item \label{enu:strongSharpLip}If $A\subseteq X$, $A$ internal, sharply
bounded and convex, then $f$: $A\to\Rtil^{d}$ is Lipschitz 
\item \label{enu:derivativeZero} If $x\in\mathrm{int}_{\text{s}}(X)$ and
$f(y)=0$, for each $y$ in a sharp neighborhood of $x$, then $\partial^{\alpha}u_{\eps}(x)\sim0$,
$\forall\alpha\in\N^{n}$. 
\item \label{enu:indepReprDerivative} If $X$ is sharply open and $f=[v_{\eps}(-)]|_{X}$,
then for each $\alpha\in\N^{n}$, $[\partial^{\alpha}u_{\eps}(-)]=[\partial^{\alpha}v_{\eps}(-)]$.
\end{enumerate}
\end{thm}
\begin{proof}
\ref{enu:locLipSharp}: by the inequality \eqref{eq:GSFLipschitz},
$f$ is Lipschitz on $B_{[\eps^{N}]}([x_{\eps}])$ for some $N$.

\ref{enu:strongSharpLip}: similar to Thm.\ \ref{prop:indepRepr},
now applying Lemma \ref{lemma:ModerateUniform} to $A_{n}=[A_{\eps}]$
with $A_{\eps}$ convex (by Lemma \ref{lem:convex}).

\ref{enu:derivativeZero}: let $N\in\N$ such that $y\in X$ and $f(y)=0$,
for each $y\in\GenR$ with $\abs{y-x}\le\eps^{N}$. Let $x=[x_{\eps}]$.
We have that $(\sup_{\abs{x-x_{\eps}}\le\eps^{N}}\abs{u_{\eps}(x)})_{\eps}\sim0$,
since otherwise, one constructs $y\in X$ with $\abs{y-x}\le\eps^{N}$
and $f(x)\ne0$. Again by Lemma \ref{lemma:ModerateUniform}, we find
for each $\alpha\in\N^{d}$ some $N\in\N$ such that $\sup_{\abs{x-x_{\eps}}\le\eps^{N}}\abs{\partial^{\alpha}u_{\eps}(x)}\le\eps^{-N}$,
for small $\eps$. The statement now follows similar to \cite[Thm.~1.2.3]{GKOS}.

\ref{enu:indepReprDerivative}: apply \ref{enu:derivativeZero} to
$[u_{\eps}-v_{\eps}]$.
\end{proof}
Consequently, the partial derivative $\partial^{\alpha}f:=[\partial^{\alpha}u_{\eps}(-)]$
on a sharply open $X\subseteq\Rtil^{n}$ is itself a well-defined
GSF, and thus satisfies itself the Lipschitz conditions from the previous
theorem. We can now show the relationship between GSF and the discontinuous
Colombeau differential calculus developed in \cite{AFJ}:
\begin{prop}
\label{prop:relWithB}Let $a,b\in\R$ with $a<b$ and $\widetilde{(a,b)}\subseteq U\subseteq\Rtil$,
where $U$ is open in the sharp topology. Let $f:U\longrightarrow\Rtil$
be a generalized smooth map, then at every point $x\in\widetilde{(a,b)}$
the function $f$ is differentiable in the sense of \cite{AFJ} with
derivative $f'(x)$.\end{prop}
\begin{proof}
\noindent Let $f$ be defined by the net of smooth functions $(u_{\eps})$.
Since $\widetilde{(a,b)}\subseteq U\subseteq\sint{\Omega_{\eps}}$,
Prop.\ \ref{prop:OmegaEpsAroundBoundedAndCptSupp} yields $(a,b)\subseteq\Omega_{\eps}$
for $\eps$ small. For all these $\eps$ and for $y\in\widetilde{(a,b)}$,
applying to $u_{\eps}$ the second order Taylor formula, we get
\begin{align*}
|f(y)-f(x)-f'(x)\cdot(y-x)|_{e} & =|[u_{\eps}(y_{\eps})-u_{\eps}(x_{\eps})-u'_{\eps}(x_{\eps})\cdot(y_{\eps}-x_{\eps})]|_{e}\\
 & =\left|\left[\frac{1}{2}u''_{\eps}(\zeta_{\eps})\cdot(y_{\eps}-x_{\eps})^{2}\right]\right|_{e}\\
 & \le|[u''_{\eps}(\zeta_{\eps})]|_{e}\cdot|(y_{\eps}-x_{\eps})|_{e}^{2},
\end{align*}
where $\zeta_{\eps}\in[x_{\eps},y_{\eps}]\cup[y_{\eps},x_{\eps}]\subseteq(a,b)\subseteq\Omega_{\eps}$.
The moderateness of $[u''_{\eps}(\zeta_{\eps})]\in\Rtil$ (condition
\ref{enu:partial-u-moderate} of Def.\ \ref{def:netDefMap}) at $[\zeta_{\eps}]\in\widetilde{(a,b)}\subseteq U\subseteq\sint{\Omega_{\eps}}$
yields $|[u''_{\eps}(\zeta_{\eps})]|_{e}\le K$ for some $K\in\R$.
Thus
\[
\lim_{y\to x}\frac{|f(y)-f(x)-f'(x)\cdot(y-x)|_{e}}{|(y-x)|_{e}}=0,
\]
as claimed.
\end{proof}
\begin{defn}
Let $X\subseteq\GenR^{d}$. We call
\begin{align*}
\EMod(X,\Rtil^{n})=\, & \big\{(u_{\eps})\in\cinfty(\R^{d},\R^{n})^{I}\mid\forall\alpha\in\N^{d}\,\forall[x_{\eps}]\in X:\ (\partial^{\alpha}u_{\eps}(x_{\eps}))\in\R_{M}^{n}\big\},\\
\Null(X,\Rtil^{n})=\, & \big\{(u_{\eps})\in\EMod(X,\Rtil^n)\mid\forall[x_{\eps}]\in X:\ (u_{\eps}(x_{\eps}))\sim0\big\}.
\end{align*}
\end{defn}
Since $(u_\eps)$ defines a smooth map $X\to\Rtil^n$ iff $(u_\eps)\in\EMod(X,\Rtil^n)$ and $[u_\eps(-)]=[v_\eps(-)]$ iff $(u_\eps - v_\eps)\in \Null(X,\Rtil^n)$, we can identify $\Gcinf(X,\Rtil^{n})$ with $\EMod(X,\Rtil^{n})/\Null(X,\Rtil^{n})$.

As a result of Thm. \ref{thm:indepRepreAndLipCond},
we can also write
\[\Null(X,\Rtil^{n})= \big\{(u_{\eps})\in\EMod(X,\Rtil^n)\mid\forall\alpha\in\N^{d}\,\forall[x_{\eps}]\in X:\ (\partial^{\alpha}u_{\eps}(x_{\eps}))\sim0\big\}
\]
if $X$ is sharply open.

\begin{rem}
\label{rem:defGenSmoothFunct}\ 
\begin{enumerate}[leftmargin=*,label=(\roman*),align=left ]
\item In general, if the GSF $f:X\longrightarrow Y$ is defined by the
net $u_{\eps}\in\cinfty(\Omega_{\eps},\R^{d})$, the function $f$
may not be extensible to the whole of $\sint{\Omega_{\eps}}\supseteq X$
because some derivative $(\partial^{\alpha}u_{\eps}(-))$ can grow
stronger than polynomially on $\sint{\Omega_{\eps}}\setminus X$.
A simple example is given by $u(x):=e^{x}$ even for domains
$\Omega_{\eps}$ such that $X=\Rtil_{c}\subseteq\sint{\Omega_{\eps}}$.
In fact, Th.\ \ref{prop:OmegaEpsAroundBoundedAndCptSupp} yields
the existence of a sequence $(\eps_{n})_{n}\downarrow0$ such that
$[n-1,n+1]\subseteq\Omega_{\eps}$ for $\eps\in(0,\eps_{n}]$. Therefore,
the point $x$ defined by $x_{\eps}:=n$ for $\eps\in(\eps_{n+1},\eps_{n}]$
lies in $\sint{\Omega_{\eps}}\setminus\Rtil_{c}$, but $u(x_{\eps})=(e^{x_{\eps}})\notin\Rtil_{M}$.
This is a necessary limitation of this approach to generalized functions:
indeed, it is not difficult to prove that the only ordered quotient
ring where infinitesimals and order are accessible (i.e.\ defined
similarly to $\Rtil$, see \cite{Gi-Ka12}) and where every smooth
operation is possible, is necessarily the Schmieden-Laugwitz one (\cite{Sc-La,Ego,Pal,Ver11}). 
\item Let $K\comp\R^{n}$ be a compact set such that $\widetilde{K}\subseteq X\subseteq\Rtil^{n}$.
Then the GSF $f:X\longrightarrow Y$ is uniquely determined on $\widetilde{K}$
by its values on near standard points (see Sec.\ \ref{sec:BasicNotions}),
i.e.\ $f=0$ on $\widetilde{K}$ iff $f(x)=0$ for all $x\in K^{\bullet}$.
In fact, suppose that $f$ vanishes on $K^{\bullet}$ but that $f(x)\ne0$
for some $x\in\widetilde{K}$. Then there exist $m\in\N$ and $(\eps_{k})_{k}\downarrow0$
such that $|u_{\eps_{k}}(x_{\eps_{k}})|>\eps_{k}^{m}$, where $(u_{\eps})$
is a net that defines $f$. Since $\left(x_{\eps_{k}}\right)_{k}$
is a sequence in the compact set $K$, we can extract a subsequence
$\bigl(x_{\eps_{k_{l}}}\bigr)_{l}$ which converges to $\bar{x}\in K$.
Set $x'_{\eps}:=x_{\eps_{k_{l}}}$ if $\eps\in(\eps_{k_{l+1}},\eps_{k_{l}}]$,
then $x'=[x'_{\eps}]\in K^{\bullet}$ since $\lim_{\eps\to0^{+}}x'_{\eps}=\lim_{l\to+\infty}x_{\eps_{k_{l}}}=\bar{x}\in K$,
but $f(x')\ne0$, a contradiction. This generalizes the analogous
property of CGF proved in \cite{KK}. 
\end{enumerate}
\end{rem}
Our next aim is to clarify the relation between CGF and GSF.
\begin{thm}
\label{cor_EModA_internal} Let $\emptyset\ne A\subseteq\GenR^{d}$
be internal and sharply bounded. Then for each sharply bounded representative
$(A_{\eps})$ of $A$, 
\begin{multline*}
\EMod(A,\Rtil^{n})=\big\{(u_{\eps})\in\cinfty(\R^{d},\R^{n})^{I}\mid\forall\alpha\in\N^{d}\,\exists N\in\N:\ \\
\sup_{x\in A_{\eps}+\eps^{N}}\abs{\partial^{\alpha}u_{\eps}(x)}\le\eps^{-N},\text{ for small }\eps\big\}.
\end{multline*}
 
\[
\Null(A,\Rtil^{n})=\big\{(u_{\eps})\in\EMod(A,\Rtil^{n})\mid \big(\sup_{x\in A_{\eps}}\abs{u_{\eps}(x)}\big)\sim0\big\}.
\]
\end{thm}
\begin{proof}
The characterization of $\EMod(A,\Rtil^{n})$ follows immediately
by the first part of Lemma \ref{lemma:ModerateUniform}. By the second
part of Lemma \ref{lemma:ModerateUniform}, it follows that $\bigl(\sup_{x\in A_{\eps}}\abs{u_{\eps}(x)}\bigr)\sim0$
for each $(u_{\eps})\in\Null(A,\Rtil^{n})$. For the converse inclusion,
let $\tilde{x}\in A$. Then $\tilde{x}=[a_{\eps}]$,
with $a_{\eps}\in A_{\eps}$ for small $\eps$. By hypothesis, $(u_{\eps}(a_{\eps}))\sim0$.
By Theorem \ref{prop:indepRepr}, $(u_{\eps}(x_{\eps}))\sim0$
for any representative $[x_{\eps}]$ of $\tilde{x}$.
\end{proof}
We have a similar characterization for more general domains:
\begin{cor}
\label{cor_GenA_union_of_internal}Let $A=\bigcup_{\lambda\in\Lambda}B_{\lambda}\subseteq\GenR^{d}$,
where each $B_{\lambda}$ is nonempty, internal and sharply bounded.
Let $(B_{\lambda,\eps})_{\eps}$ be a sharply bounded representative
of $B_{\lambda}$, for each $\lambda$. Then 
\begin{multline*}
\EMod(A,\Rtil^{n})=\big\{(u_{\eps})\in\cinfty(\R^{d},\R^{n})^{I}\mid\forall\alpha\in\N^{d}\,\forall\lambda\in\Lambda\,\\
\exists N\in\N:\ \sup_{x\in B_{\lambda,\eps}+\eps^{N}}\abs{\partial^{\alpha}u_{\eps}(x)}\le\eps^{-N},\text{ for small }\eps\big\}.
\end{multline*}
 
\[
\Null(A,\Rtil^{n})=\big\{(u_{\eps})_{\eps}\in\EMod(A,\Rtil^{n})\mid\forall\lambda\in\Lambda:\ \big(\sup_{x\in B_{\lambda,\eps}}\abs{u_{\eps}(x)}\big)\sim0\big\}.
\]
\end{cor}
\begin{proof}
This follows by Theorem \ref{cor_EModA_internal} because, by definition,
$\EMod(\bigcup_{\lambda\in\Lambda}B_{\lambda},\Rtil^{n})=\bigcap_{\lambda\in\Lambda}\EMod(B_{\lambda},\Rtil^{n})$
and $\Null(\bigcup_{\lambda\in\Lambda}B_{\lambda},\Rtil^{n})=\bigcap_{\lambda\in\Lambda}\Null(B_{\lambda},\Rtil^{n})$.\end{proof}
\begin{thm}
\label{cor_Gen_Omega_is_tGen_Omega_c}Let $\Omega$ be an open subset
of $\R^{d}$. Then $\gs(\Omega)=\tGen(\widetilde{\Omega}_{c})$.\end{thm}
\begin{proof}
Any $u\in\gs(\Omega)$ has a representative $(u_{\eps})\in\cinfty(\R^{d})^{I}$
by the cut-off procedure in Lemma \ref{lem:fromOmega_epsToRn}. Since
$\widetilde{\Omega}_{c}=\bigcup_{K\comp\Omega}\widetilde{K}$, the
result follows by Corollary \ref{cor_GenA_union_of_internal}.
\end{proof}
Similarly, since $\Rtil^{d}=\bigcup_{n\in\N}\{x\in\Rtil^{d}:\abs{x}\le[\eps]^{-n}\}$,
$\tGen(\Rtil^{d})$ coincides with the definition of $\gs(\Rtil^{d})$
given in \cite{Ver09}, where it is also shown that $\gs_{\tau}(\R^{d})\subseteq\tGen(\Rtil^{d})$.
\begin{rem}
Thus essentially, GSF have a greater flexibility in their domains
compared with CGF, which always have a domain of the form $\otilc$.
\begin{enumerate}[leftmargin=*,label=(\roman*),align=left ]
\item The possibility to define a GSF using a net $u_{\eps}\in\cinfty(\Omega_{\eps},\R^{d})$,
permits to obtain GSF which are defined on purely infinitesimal sets,
e.g.\ starting from $\Omega_{\eps}=(-\eps^{q},\eps^{q})$, so that
we can take $X=\sint{\Omega_{\eps}}\subseteq B_{[\eps^{q}]}(0)\subseteq\Rtil$. 
\item Vice versa, we can define GSF on unbounded sets of generalized points.
A simple case is the exponential map 
\[
e^{(-)}:x\in\left\{ x\in\Rtil\mid\exists z\in\Rtil_{>0}^{*}:\ |x|\le|\log z|\right\} \mapsto e^{x}\in\Rtil.
\]
 The domain of this GSF cannot be of the form $\widetilde{\Omega}_{c}$
(which contains only finite points). Analogously, the domain of the
map 
\[
e^{\frac{1}{(-)}}:\left\{ x\in\Rtil\mid\exists z\in\Rtil_{>0}^{*}:\ |x^{-1}|\le|\log z|\right\} \mapsto e^{\frac{1}{x}}\in\Rtil
\]
 contains a set of infinitesimals that is not of the form $\otilc$. 
\end{enumerate}
\end{rem}
Contrary to the case of distributions and CGF, there is no problem
in considering the composition of two GSF:
\begin{thm}
\label{thm:GSFcategory} Subsets $S\subseteq\Rtil^{s}$ with the trace
of the sharp topology, and generalized smooth maps as arrows form
a subcategory of the category of topological spaces. We will call
this category $\Gcinf$, the \emph{category of GSF}.\end{thm}
\begin{proof}
By Theorem \ref{thm:indepRepreAndLipCond}.\ref{enu:locLipSharp}
we already know that every GSF is continuous; we have hence to prove
that these arrows are closed with respect to identity and composition
in order to prove that we have a concrete subcategory of topological
spaces and continuous maps.

\noindent If $T\subseteq\Rtil^{t}$ is an object, then $u_{\eps}(x):=x$
is the net of smooth functions that globally defines the identity
$1_{T}$ on $T$. It is immediate that $1_{T}$ is generalized smooth.

\noindent To prove that arrows of $\Gcinf$ are closed with respect
to composition, let $S\subseteq\Rtil^{s},T\subseteq\Rtil^{t},R\subseteq\Rtil^{r}$
and $f=[u_{\eps}(-)]:S\longrightarrow T$, $g=[v_{\eps}(-)]:T\longrightarrow R$
be generalized smooth maps, where we can choose $u_{\eps}\in\cinfty(\R^{s},\R^{t})$
and $v_{\eps}\in\cinfty(\R^{t},\R^{r})$ by Lemma \ref{lem:fromOmega_epsToRn}.
Then $v_{\eps}\circ u_{\eps}\in\cinfty(\R^{s},\R^{r})$. We show that
$(v_{\eps}\circ u_{\eps})_{\eps}$ defines the GSF $v\circ u$: $S\longrightarrow R$.\\
 For every $x=[x_{\eps}]\in S$, $f(x)=[u_{\eps}(x_{\eps})]\in T$
and thus $g(f(x))=[v_{\eps}(u_{\eps}(x_{\eps}))]\in R$. Consider
any $\gamma\in\N^{s}$. It remains to be shown that $\partial^{\gamma}(v_{\eps}\circ u_{\eps})(x_{\eps})\in\R_{M}^{r}$.
We can write 
\begin{equation}
\partial^{\gamma}(v_{\eps}\circ u_{\eps})(x_{\eps})=p\left[\partial^{\alpha_{1}}u_{\eps}(x_{\eps}),\ldots,\partial^{\alpha_{A}}u_{\eps}(x_{\eps}),\partial^{\beta_{1}}v_{\eps}(u_{\eps}(x_{\eps})),\ldots,\partial^{\beta_{B}}v_{\eps}(u_{\eps}(x_{\eps}))\right],\label{eq:derivativeOfComposition}
\end{equation}
 where $p$ is a suitable polynomial not depending on $x_{\eps}$.
Every term $\partial^{\alpha_{i}}u_{\eps}(x_{\eps})$ and $\partial^{\beta_{j}}v_{\eps}(u_{\eps}(x_{\eps}))$
is moderate by \ref{enu:partial-u-moderate} of Def.\ \ref{def:netDefMap}.
Since moderateness is preserved by polynomial operations, it follows
that also $\partial^{\gamma}(v_{\eps}\circ u_{\eps})(x_{\eps})$ is
moderate.\end{proof}

\end{document}